\documentclass{article}
\usepackage{amsthm}
\usepackage{amsmath}
\usepackage{amssymb}
\usepackage{booktabs}
\usepackage{setspace} 
\usepackage{enumerate}

\newtheorem{prop}{Proposition}
\newtheorem{rem}{Remark}
\newtheorem{thm}{Theorem}
\newtheorem{defn}{Definition}
\newtheorem{lem}{Lemma}
\newtheorem{cor}{Corollary}

\DeclareMathOperator*{\argmin}{arg\:min}
\DeclareMathOperator{\TV}{TV}
\DeclareMathOperator{\BV}{BV}
\DeclareMathOperator{\dive}{div}
\DeclareMathOperator{\subdif}{\partial}
\DeclareMathOperator{\domain}{\text{dom}}

\DeclareMathOperator{\supp}{supp}

\newcommand{\Du}{ \mathrm{D} u }
\newcommand{\vDu}{\vert \mathrm{D} u \vert}
\newcommand{\wrt}{\:\mathrm{d}}
\newcommand{\Wrt}{\mathrm{D}}
\newcommand{\limn}{\underset{n\rightarrow \infty }{\lim}}
\newcommand{\wdiv}{W^q(\dive;\Omega)}

\newcommand{\ones}{\mathbf{1}}

\newcommand{\R}{\mathbb{R}}
\newcommand{\Rtwo}{\mathbb{R}^2}

\title{A pointwise characterization of the subdifferential of the total variation functional\footnote{Support by the special research grant SFB ``Mathematical Optimization and Applications in Biomedical Sciences'' of the Austrian Science Fund (FWF) is gratefully acknowledged.}}

\author{K. Bredies\footnote{Department of Mathematics, University of Graz, Heinrichstr. 36, A-8010 Graz, Austria. Mail: kristian.bredies@uni-graz.at, Phone: +43 316 380 5158.} \and M. Holler\footnote{Corresponding author. Department of Mathematics, University of Graz, Heinrichstr. 36, A-8010 Graz, Austria. Mail: martin.holler@uni-graz.at, Phone: +43 316 380 5156.}}
\date{}

\begin{document}
\maketitle

\begin{abstract}
We derive a new pointwise characterization of the subdifferential of the total variation ($ \TV $) functional. It involves a full trace operator which maps certain $ L^q $ - vectorfields to integrable functions with respect to the total variation measure of the derivative of a bounded variation function. This full trace operator extents a notion of normal trace, frequently used, for example, to characterize the total variation flow.
\end{abstract}


 \newpage

\section{Introduction}
The aim of this paper is to derive a new, pointwise characterization of the subdifferential of the $ \TV $ functional in Lebesgue spaces.
This characterization bases on a trace operator, which extends the normal trace of \cite{Anzellotti83}: There, Anzellotti introduces a normal trace $ \theta (g,\mathrm{D}u)\in L^1(\Omega;\vDu)$ for vector fields $ g\in W^q (\dive ;\Omega ) \cap L^\infty (\Omega,\mathbb{R}^d) $ (see Section \ref{sec:tools}) that allows the following characterization:
$ u^*\in \partial \TV (u) $ if and only if, there exists $ g \in W^q _0 (\dive ;\Omega ) $ with $ \Vert g\Vert _\infty \leq 1 $ such that $ u^* = -\dive g $ and  \[ \theta (g,\mathrm{D}u)=1 \quad  \mathrm{in}\,  L^1 (\Omega ;\vDu) .\]
This approach is commonly used to characterize the total variation flow, as for example in \cite{Andreu01dirichlet,Andreu01,Andreu02,Andreu09,Bellettini02,Bellettini05,Burger07}.

Introducing a ``full'' trace operator $ T:D\subset W^q  (\dive;\Omega  ) \cap L^\infty (\Omega, \mathbb{R}^d) \rightarrow L^1 (\Omega,\mathbb{R}^d ; \vDu ) $, we sharpen this result by showing that the set $\partial \TV (u) $ can be described as:
$ u^*\in \partial \TV (u) $ if and only if, there exists $ g \in D\cap W^q _0 (\dive;\Omega  ) $ with $ \Vert g\Vert _\infty \leq 1 $ such that $ u^* = -\dive g $ and
\[Tg = \sigma_u  \quad \mathrm{in} \, L^1 (\Omega , \mathbb{R}^d ;\vDu ),\]
where $ \sigma _u \in L^1 (\Omega , \mathbb{R}^d ; \vDu )  $ is the density function such that $ \mathrm{D} u = \sigma _u \vDu $.

The outline of the paper is as follows: In the second section we give some preliminary results about functions of bounded variation, introduce a straightforward generalization of the space $ H(\dive) $ and state an approximation result. The third section is the main section, where we first repeat the term of normal trace introduced in \cite{Anzellotti83}, then introduce the notion of full trace, and, using this notion, show a characterization of the subdifferential of the total variation ($ \TV $) functional. In the fourth section we address some topics where the full trace characterization of the $ \TV $ subdifferential can be applied:  We use it to reformulate well known results, such as a characterization of the total variation flow, a characterization of Cheeger sets and optimality conditions for mathematical imaging problems, in terms of the full trace operator.  In the last section we give a conclusion.
\section{Preliminaries}
\label{sec:tools}

This section is devoted to introduce notation and basic results. After some preliminary definitions, we start with a short introduction to functions of bounded variation. For further information and proofs we refer to to \cite{Ambrosio, Ziemer, Evans}. For convenience, we always assume  $ \Omega \subset \mathbb{R}^d  $ to be a bounded Lipschitz domain. Further, throughout this work, we often denote $ \intop _\Omega \phi $ or $ \intop _\Omega \phi \wrt x $ instead of $ \intop _\Omega \phi(x) \wrt x $ for the Lebesgue integral of a measureable function $ \phi $, when the usage of the Lebesgue measure and the integration variable are clear from the context.

We use a standard notation for continuously differentiable-, compactly supported- or integrable functions. However, in order to avoid ambiguity, we define the space of continuously differentiable functions on a closed set:
\begin{defn}[Continuous functions on a closed set]
Given a domain $ A\varsubsetneqq \R ^d $ and $ m\in \mathbb{N} $, we define
\[C(\overline{A},\R ^m) = \{ \phi :\overline{A}\rightarrow \R^m \,|\, \phi \mbox{ is uniformly continuous on A} \}, \]
\[C^k(\overline{A},\R ^m) = \{ \phi :\overline{A}\rightarrow \R^m \,|\, \Wrt ^\alpha \phi \in C(\overline{A},\R ^m) \mbox{ for all } |\alpha | \leq k \} \]
and
\[ C^\infty (\overline{A},\R ^m) = \bigcap _{k\in \mathbb{N}} C^k(\overline{A},\R ^m) .\]
\end{defn}
Note that for bounded domains, $\phi \in C(\overline{A}, \R^m)$ is equivalent
to $\phi$ being the restriction of a function in $C_c(\R^d, \R^m)$. This
also applies to $C^k(\overline{A},\R^m)$ and $C^\infty(\overline{A},\R^m)$
with $C_c^k(\R^d, \R^m)$ and $C_c^\infty(\R^d,\R^m)$, respectively, by
virtue of Whitney's Extension Theorem \cite[Theorem 1]{Whitney34}. 
For unbounded domains, however, this is generally not true.
\begin{defn}[Finite Radon measure]
Let  $ \mathcal{B}(\Omega) $ be the Borel $ \sigma $-algebra generated by the open subsets of $ \Omega $. We say that a function $ \mu : \mathcal{B}(\Omega) \rightarrow \mathbb{R}^m $, for $ m\in \mathbb{N} $, is a finite $ \mathbb{R}^m $-valued Radon measure if $ \mu (\emptyset) = 0 $ and $ \mu $ is $ \sigma $-additive. We denote by $ \mathcal{M}(\Omega) $ the space of all finite Radon measures on $ \Omega $. Further we denote by $ |\mu | $ the variation of $ \mu \in \mathcal{M}(\Omega) $, defined by \[
\vert\mu\vert(E)=\sup\left\{ \sum_{i=0}^{\infty}\vert\mu(E_{i})\vert\,\Bigl|\,E_{i}\in \mathcal{B}(\Omega),\, i\geq 0 ,\,\mbox{ pairwise disjoint,}\, E=\bigcup_{i=0}^{\infty}E_{i}\right\}, \] for $ E\in \mathcal{B}(\Omega) $. Note that $ |\mu (E_i) | $ denotes the Euclidean norm of $ \mu (E_i) \in \R ^m $.
\end{defn}

\begin{defn}[Functions of bounded variation]
We say that a function $u\in L^1(\Omega)$ is of bounded variation, if there exists a finite $\mathbb{R}^{d}$-valued
Radon measure, denoted by $Du=(D_{1}u,...,D_{d}u)$, such that for
all $i\in\{1,...,d\}$, $D_{i}u$ represents the distributional derivative
of $u$ with respect to the $i$th coordinate, i.e., we have\[
\intop_{\Omega}u\partial_{i}\phi=-\intop_{\Omega}\phi\wrt D_{i}u\quad\mbox{for all }\phi\in C_{c}^{\infty}(\Omega).\]
By $\BV(\Omega)$ we denote the space of all functions $ u\in L^1(\Omega) $ of bounded
variation.\end{defn}
\begin{defn}[Total variation]
For $u\in L^{1}(\Omega)$, we define the functional
$\TV:L^{1}(\Omega)\rightarrow\overline{\mathbb{R}}$ as \[
\TV(u)=\sup\left\{ \intop_{\Omega}u\dive\phi\,\Bigg\vert\,\phi\in C_{c}^{\infty}(\Omega,\mathbb{R}^{d}),\,\Vert\phi\Vert_{\infty}\leq1\right\} \]
where we set $\TV(u)=\infty$ if the set is unbounded from above.
We call $\TV(u)$ the total variation of $u$. 
\end{defn}
\begin{prop}
The functional $ \TV: L^1(\Omega) \rightarrow \overline{\mathbb{R}} $ is convex and lower semi-continuous with respect to $ L^1 $-convergence. For $u\in L^{1}(\Omega)$ we
have that \[u\in\BV(\Omega) \mbox{ if and only if } \TV(u)<\infty.\] In addition,
the total variation of $u$ coincides with the variation of the measure
$Du$, i.e., $\TV(u)=\vert Du\vert(\Omega)$.
Further, \[\Vert u\Vert_{\BV}:=\Vert u\Vert_{L^{1}}+\TV(u)\] defines
a norm on $\BV(\Omega)$ and endowed with this norm, $\BV(\Omega)$
is a Banach space. \end{prop}
\begin{defn}[Strict Convergence]
For $ (u_n)_{n\in \mathbb{N}} $ with $ u_n \in \BV (\Omega) $, $ n\in \mathbb{N} $, and $ u\in \BV (\Omega) $ we say that $ (u_n)_{n\in \mathbb{N}} $ strictly converges to $ u $ if \[ \| u_n - u\| _{L^1} \rightarrow 0 \mbox{ and } \TV (u_n ) \rightarrow \TV (u) \] as $ n\rightarrow \infty $.
\end{defn}

\begin{defn}[Lebesgue Point]
Let $ f\in L^p ( \Omega)$, $ 1\leq p < \infty  $. We say that $ x\in \Omega $ is a Lebesgue point of $ f $ if 
\[\underset{r\rightarrow 0}{\lim} \frac{1}{|B(x,r)|} \intop _{B(x,r)} | f(y) - f(x) | \wrt y \rightarrow 0 \]
as $ n\rightarrow \infty  $. Note that here, $ |B(x,r)|$ denotes the Lebesgue measure of the ball with radius $ r $ around $ x\in \Omega $.

\end{defn}
\begin{rem} Remember that for any $f\in L^p (\Omega) $, $ 1\leq p < \infty  $, almost every $ x \in \Omega $ is a Lebesgue point of $ f $ (see \cite[Corollary 1.7.1]{Evans}).

\end{rem}
Next we recall some standard notations and facts from convex analysis. For proofs and further introduction we refer to \cite{Ekeland}.

\begin{defn}[Convex conjugate and subdifferential]
\label{defn:polar}
For a normed vector space $ V$ and a function $ F:V\rightarrow \overline{\mathbb{R}}$ we define its convex conjugate, or Legendre-Fenchel transform, denoted by $F^* : V^* \rightarrow \overline{\mathbb{R}}$, as \[F^*(u^*) = \underset{v\in V}{\sup} \langle v,u^* \rangle_{V,V^*} - F(v) .\] Further $ F $ is said to be subdifferentiable at $ u\in V $ if $ F(u) $ is finite and there exists $ u^* \in V^* $ such that \[\langle v-u,u^*\rangle_{V,V^*} + F(u) \leq F(v) \] for all $ v\in V $. The element $u^* \in V^* $ is then called a subgradient of $ F $ at $ u $ and the set of all subgradients at $ u $ is denoted by $ \subdif F(u) $.
\end{defn}
\begin{defn}[Convex indicator functional]
 For a normed vector space $ V $ and $ U\subset V $ a convex set, we denote by $ \mathcal{I}_U : V \rightarrow \overline{\mathbb{R}} $ the convex indicator functional of $ U $, defined by \[ \mathcal{I}_U (u) = \begin{cases} 0 & \mbox{ if } u\in U, \\ \infty &\mbox{ else.} \end{cases} \]
\end{defn}

Next we define the space $ W^q(\dive;\Omega) $, which is fundamental for the characterization of the $ \TV $ subdifferential.
\begin{defn}[The space $ W^q(\dive;\Omega) $]
\label{def:Wq(div)}Let $ 1\leq q <\infty  $ and $g\in L^{q}(\Omega,\mathbb{R}^{d})$. We say that
$\dive g\in L^{q}(\Omega)$ if there exists $w\in L^{q}(\Omega)$
such that for all $v\in C_{c}^{\infty}(\Omega)$\[
\intop_{\Omega}\nabla v \cdot g=-\intop_{\Omega}vw.\] 
Furthermore we define \begin{equation*}
W^{q}(\dive;\Omega)=\left\{ g\in L^{q}(\Omega,\mathbb{R}^{d})\,\vert\,\dive g\in L^{q}(\Omega)\right\}
\end{equation*}
with the norm $\Vert g\Vert_{W^{q}(\dive)}^{q}:=\Vert g\Vert_{L^{q}}^{q}+\Vert\dive g\Vert_{L^{q}}^{q}.$\end{defn}
\begin{rem}
Density of $C_{c}^{\infty}(\Omega)$ in $L^{p}(\Omega)$ implies that, if there exists $w\in L^{q}(\Omega)$ as above,
it is unique. Hence it makes sense to write $\dive g=w$. By completeness of $L^{q}(\Omega)$ and $L^{q}(\Omega,\mathbb{R}^{d})$ it follows that $W^{q}(\dive;\Omega)$ is a Banach space when equipped with $\Vert\cdot\Vert_{W^{q}(\dive)}$.
\end{rem}
\begin{rem} Note that $ W^q (\dive ; \Omega) $ is just a straightforward generalization of the well known space $ H(\dive ; \Omega ) $. Also classical results like density of $ C^\infty (\overline{\Omega},\mathbb{R}^d) $ and existence of a normal trace on $ \partial \Omega $ can be derived for $ W^q (\dive ;\Omega ) $ as straightforward generalizations of the proofs given for example in \cite[Chapter 1]{Girault}.
\end{rem}

\begin{defn}
\label{def:W0(div)}For $ 1\leq q < \infty $, we define \[
W_{0}^{q}(\dive;\Omega)=\overline{C_{c}^{\infty}(\Omega,\mathbb{R}^d)}^{\Vert\cdot\Vert_{W^{q}(\dive)}}.\]
\end{defn}
\begin{rem} \label{rem:gauss_green_w0_div}
By density it follows that, for $ g\in W_0 ^q \dive;\Omega) $, we have
\[\intop _\Omega \nabla v g = - \intop _\Omega v \dive g\] for all $ v\in C^\infty(\overline{\Omega}) $.
\end{rem}

The following approximation result will be needed in the context of the full trace.
\begin{prop} \label{prop:w_div_approximation}
If $ \Omega $ is a bounded Lipschitz domain, $ 1\leq q <\infty  $ and $ g\in W^q (\dive ;\Omega) $, there exists a sequence of vector fields $ (g_n)_{n\geq 0 }\subset C^\infty (\overline{\Omega},\mathbb{R}^d ) $ such that

\begin{enumerate}
\item $\Vert g_n - g \Vert _{W^q(\dive)} \rightarrow 0 \mbox{ as } n \rightarrow \infty $,
\item $\Vert g_n \Vert _\infty \leq \Vert g \Vert _\infty$ for each $ n\in \mathbb{N} $, if $ \| g \| _\infty < \infty  $,
\item $g_n (x) \rightarrow g(x) $ for every Lebesgue point $x\in \Omega $ of g.
\item $ \Vert g_n - g \Vert _{\infty,\overline{\Omega}} \rightarrow 0 $ as $ n\rightarrow \infty $, if, additionally, $ g\in C(\overline{\Omega},\mathbb{R}^d) $.
\end{enumerate}
\end{prop} A proof can be found in the Appendix.

\section{Subdifferential of TV} \label{tv_subdif}

In order to describe the subdifferential of the $ \TV $ functional, for $ u\in \BV (\Omega) $, we need a notion of trace for $ W^q (\dive ; \Omega )$ vector fields in $ L^1 (\Omega,\mathbb{R}^d ; \vDu ) $.

\subsection{The normal trace}
We first revisit the normal trace introduced in \cite{Anzellotti83}. We do so by defining it for $ W^q(\dive ; \Omega) $ vector fields as a closed operator. In this subsection, if not restricted further, let always be $ 1\leq q < \infty  $, $ p=\frac{q}{q-1} $ if $ q\neq 1 $ or $ p=\infty $ else, and $ \Omega $ a bounded Lipschitz domain.

\begin{prop}\label{prop:Dual Operator}
Set $\tilde{D}_N:=W^q  (\dive ;\Omega)\cap L^\infty (\Omega,\mathbb{R}^d)$. Then, with $u\in \BV (\Omega) \cap L^p (\Omega )$ fixed, for any $z\in \tilde{D} _N$ there exists a function $ \theta (z,\Du) \in L^1 (\Omega; \vDu) $ such that 
\[\intop _\Omega  \theta(z,\Du) \psi \wrt \vDu = - \intop _\Omega u \dive (z\psi) \wrt x \]
for all $ \psi \in C^\infty _c (\Omega ) $.
\begin{proof}
For $ z\in \tilde{D}_N $ we define
\begin{eqnarray*}
L_z:C^{\infty}_c (\Omega) & \rightarrow & \mathbb{R}\\
\psi & \mapsto & -\intop_{\Omega}u\dive(z\psi)\wrt x
\end{eqnarray*}
and show that $ L_z $ can be extended to a linear, continuous operator from $C_0 (\Omega)$ to $\mathbb{R}$.

It is clear that $L_z$ is well-defined and linear, hence by definition of $C_0 (\Omega) $ as closure of $ C_c ^\infty (\Omega) $ with respect to $ \| \cdot \| _\infty $, it suffices to show that $L_z$ is continuous with respect to $\Vert \cdot \Vert _\infty $. With $ \psi \in C^\infty _c (\Omega) $ and $ (z_n)_{n\geq 0} \subset C^\infty (\overline{\Omega},\mathbb{R}^d) $ converging to $ z $ as in Proposition \ref{prop:w_div_approximation}, we estimate

\begin{eqnarray*}
\vert L_z (\psi) \vert & = & \limn \bigg \vert -\intop _\Omega u \dive (z_n \psi)\wrt x \bigg \vert  =  \limn \bigg \vert \intop _\Omega z_n \psi  \wrt \Wrt u\bigg \vert  \\
& \leq &  \Vert z \Vert _\infty  \intop _\Omega  \vert \psi \vert  \wrt  \vert \Wrt u \vert \leq \Vert z \Vert _\infty \Vert \psi \Vert _\infty \vDu (\Omega) ,
\end{eqnarray*}
where we used that $ \Vert z_n - z \Vert _{W^q(\dive )} \rightarrow 0 $ as $ n\rightarrow \infty $ and that $ \Vert z_n \Vert _\infty \leq \Vert z \Vert _\infty $ for each $ n\in \mathbb{N} $.

Thus, for any $ z\in W^q  (\dive ;\Omega)\cap L^\infty (\Omega,\mathbb{R}^d) $, we have that $ L_z\in C_0 (\Omega) ^*=\mathcal{M}(\Omega) $ and we can write $ (z,\Wrt u) $ for the Radon measure associated with $ L_z $. Performing the above calculations for $ \psi \in C_c ^\infty (A) $ with any open $ A\subset \Omega  $ yields $ \vert L_z (\psi )\vert \leq \Vert z \Vert _\infty \Vert \psi \Vert _\infty \vDu (A) $. Thus it follows that $ (z,\Wrt u)\ll \vDu  $ and hence by the Radon-Nikodym theorem there exists $ \theta (z,\Wrt u)\in L^1 (\Omega ; \vert \Wrt u\vert ) $ such that $(z,\Wrt u)=\theta (z,\Wrt u)\vert \Wrt u\vert $.
\end{proof}
\end{prop}

With that we can define the normal trace operator and prove additional properties:
\begin{prop}[Normal trace operator]\label{prop:normal_trace_operator}
With $ \tilde{D}_N $ as in Proposition \ref{prop:Dual Operator} and $ u\in \BV (\Omega) \cap L^p (\Omega)  $ fixed, the operator
\begin{eqnarray*}
\widetilde{T_N}:\tilde{D} _N\subset W^q (\dive;\Omega )&\rightarrow &L^1(\Omega;\vert \Wrt u\vert )\\
z&\mapsto & \theta (z,\Wrt u)
\end{eqnarray*}
with $ \theta (z,\Wrt u) $ the density function of the measure $ (z,\Wrt u) $ with respect to $ \vDu $ as above, is well-defined and closeable. Further, with $ T_N: D_N \rightarrow L^1(\Omega;\vert \Wrt u\vert ) $ denoting the closure of  $ \widetilde{T_N} $ defined on $ D_N \subset W^q(\dive;\Omega)$, we have that, for $ z\in D_N $, 
\[ \Vert T_N z \Vert _\infty \leq \Vert z \Vert _\infty  \] 
whenever $ z\in L^\infty(\Omega , \mathbb{R}^d ) $ and, for $ \phi \in C(\overline{\Omega},\mathbb{R}^d ) \cap W^q (\dive;\Omega ) $, that 
\[ T_N\phi= \phi \cdot \sigma _u \in L^1 (\Omega ;\vert \Wrt u\vert )\]
 where $ \sigma _u $ is the density function of $ \Wrt u $ w.r.t. $ \vDu $.

\begin{proof}
Well-definition is clear since the representation of $ L_z $ as a measure and also its density function with respect to $ \vDu $ is unique. Let now $ (z_n)_{n\geq 0},(\tilde{z}_n)_{n\geq 0} \subset \tilde{D} _N $ be two sequences converging to $ z $ in $ W^q (\dive;\Omega ) $ and suppose that $ \widetilde{T_N} z_n \rightarrow h $ and $ \widetilde{T_N} \tilde{z}_n \rightarrow \tilde{h} $ with $ h,\tilde{h} \in L^1 (\Omega;\vert \Wrt u \vert ) $. With $ \psi \in C^\infty _c (\Omega) $ we can write, using $ \limn \dive (z_n \psi ) = \dive (z\psi) = \limn \dive (\tilde{z}_n \psi ) $ in $ L^q (\Omega) $,
\begin{eqnarray*}
\intop _\Omega h \psi \wrt \vDu & = & \limn \intop _\Omega (\widetilde{T_N} z_n) \psi \wrt \vDu  =   \limn -\intop_\Omega u \dive (z_n \psi) \wrt x \\ & = & \limn -\intop _\Omega u \dive (\tilde{z}_n \psi) \wrt x  =  \limn \intop _\Omega (\widetilde{T_N} \tilde{z}_n) \psi \wrt \vDu \\ & = & \intop _\Omega \tilde{h} \psi \wrt \vDu
\end{eqnarray*}
and thus, by density, $ h=\tilde{h} $ and, consequently, $ \widetilde{T_N} $ is closeable. The assertion $ \Vert T_N z \Vert _\infty  \leq \Vert z \Vert _\infty $ for $ z\in D_N $ follows from $ \left| \intop _A \theta(z,\Du) \wrt \vDu \right| \leq \Vert z \Vert _\infty \vDu (A) $, for all $ A\subset \Omega $ measurable, in the case that $ \Vert z\Vert _\infty  <\infty $, since then $ z\in \tilde{D}_N $. If $ \| z \|_\infty = \infty$, the inequality is trivially satisfied.

In order to show that $ T_N \phi =\phi \cdot \sigma _u $ for $ \phi \in C(\overline{\Omega},\mathbb{R}^d)\cap W^q (\dive;\Omega ) $ first note that $ \phi \in \tilde{D}_N$. Thus, $ T_N \phi $ is defined and we can use that, due to continuity of $ \phi $, the approximating  vector fields $ (\phi_n)_{n\geq 0} $ as in Proposition \ref{prop:w_div_approximation}  converge uniformly to $ \phi $ and write, again for $ \psi \in C_c ^\infty (\Omega) $,
\begin{eqnarray*}
\intop _\Omega (T_N \phi) \psi \wrt \vDu & = & -\intop _\Omega u \dive(\phi \psi) \wrt x  =   \limn -\intop _\Omega u \dive (\phi _n \psi) \wrt x \\ & = & \limn \intop _\Omega \phi _n\psi \wrt \Wrt u  =  \intop _\Omega (\phi \cdot \sigma _u) \psi \wrt \vDu  .
\qedhere
\end{eqnarray*}
\end{proof}
\end{prop}
\begin{rem}
Note that by similar arguments one could also show that $ \widehat{T_N}: X(\Omega):=W^q (\dive;\Omega  ) \cap L^\infty (\Omega,\mathbb{R}^d) \rightarrow L^1(\Omega;\vDu) $ is continuous, when $ X $ is equipped with the norm $ \Vert z \Vert _X := \Vert z \Vert _\infty +  \Vert \dive z \Vert _{L^q}$.
\end{rem}
We therefore have a suitable notion of normal trace for a dense subset of $ W^q (\dive;\Omega ) $. The closedness of the operator $ T_N $ can be interpreted as follows: If $ z\in W^q (\dive;\Omega )\cap L^\infty (\Omega,\mathbb{R}^d) $ is sufficiently regular in the sense that the normal trace of its approximating vector fields as in Proposition \ref{prop:w_div_approximation} converges to some $ h\in L^1(\Omega;\vDu) $ with respect to $ \Vert \cdot \Vert _{L^1} $ (which is satisfied for example if $ z_n $ converges pointwise $ \vDu $-a.e.), then $ T_N z=h=\limn (z_n \cdot \sigma _u) $ with $ \sigma _u $ again the density function of $ \Wrt u $ with respect to $ \vDu $.

\subsection{The full trace}
As we can see in Proposition \ref{prop:normal_trace_operator} the normal trace only provides information about the vector field $ g $ in the direction $ \sigma_u $. In the following we introduce a notion of trace which gives full vector information $ \vDu $-a.e. As for the normal trace, we also define the full trace for a dense subset of $ W^q(\dive ;\Omega) $-vector fields, where again, throughout this subsection, we assume that $ 1\leq q<\infty $. As we will see, existence of a full trace is a stronger condition than existence of a normal trace as above. Moreover, the full trace extends the notion of normal trace in the following sense: If for $ g\in W^q  (\dive;\Omega ) \cap L^{\infty} ( \Omega ,\mathbb{R} ^d) $ there exists a full trace $ h\in L^1 (\Omega,\mathbb{R}^d ; \vDu ) $, this implies that the normal trace $ T_N g $ can be written as $ T_N g=h \cdot \sigma _u $. 
First we need to define a notion of convergence:
\begin{defn}\label{defn:notion_of_convergence}
Let $ g\in W^q  (\dive ;\Omega )  \cap L^\infty (\Omega,\mathbb{R}^d). $ For $ (g_n)_{n\geq 0} \subset C (\overline{\Omega},\mathbb{R}^d )\cap \wdiv $ we say that $ (g_n)_{n\geq 0} \overset{\sim}{\rightarrow} g $ if

\begin{enumerate}
\item $\Vert g_n -g \Vert _{W^q (\dive)} \rightarrow 0$,
\item $\Vert g_n \Vert _\infty \leq \Vert g \Vert _\infty$,
\item $g_n (x) \rightarrow g(x) \text{ for every Lebesgue point x of }g.$
\end{enumerate}

\end{defn}

Note that by Proposition \ref{prop:w_div_approximation}, for every $ g\in W^q (\dive , \Omega) $ there exists a sequence $ (g_n)_{n\geq 0} \subset C^\infty (\overline{\Omega},\mathbb{R}^d) $ converging to $ g $ in the above sense.

\begin{defn}[Full trace operator]\label{defn:Full_trace}
With $ u\in \BV (\Omega ) $, define \[ T:  D \subset W^q  (\dive ; \Omega ) \cap L^{\infty} (\Omega,\mathbb{R}^d ) \rightarrow L^1 ( \Omega,\mathbb{R}^d ; \vDu ) \] by

\begin{equation*}
v=Tg 
\end{equation*}
whenever
\begin{eqnarray} \label{tg_v}
\left\{
\begin{gathered}
\text{for all } (g_n)_{n\geq 0} \subset  C^\infty (\overline{\Omega},\mathbb{R}^d) \text{ such that } g_n \overset{\sim}{\rightarrow} g,  \\ \text{it follows that } \Vert g_n - v \Vert _{L^1(\Omega,\mathbb{R}^d ; \vDu ) }\rightarrow 0,
\end{gathered}\right.
\end{eqnarray}
where
\begin{multline*} D=  \left\{ g \in W^q  (\dive;\Omega ) \cap L^{\infty}(\Omega ,\mathbb{R}^d) \, \vert \,\right. \\ \left. \text{there exists } v\in L^1 (\Omega,\mathbb{R}^d ; \vDu ) \text{ satisfying } \eqref{tg_v}  \right\}.
\end{multline*}
\end{defn}
Clearly, such $ v=Tg $ is unique in $ L^1 (\Omega,\mathbb{R}^d ; \vDu ) $ and hence $ T $ is well-defined. The next two propositions give some basic properties of the trace operator.  It is shown that $ T $ is consistent with the normal trace operator and, as one would expect, is the identity for continuous vector fields. In the following we denote by $ \vert \Wrt ^a u \vert $ the absolute continuous part of the measure $ \vDu $ with respect to $ \mathcal{L}^d $.
\begin{prop}\label{prop:basic_prop_full_trace_1}
For $ u\in \BV (\Omega) $ and $ g\in D $ with $ D $ as in Definition \ref{defn:Full_trace}, we have that\[ Tg = g \quad \vert \Wrt ^a u \vert -a.e.,\]
\[\Vert Tg \Vert _\infty \leq \Vert g \Vert _\infty.\] 
\begin{proof}
Take $ (g_n)_{n\geq 0 } \overset{\sim}{\rightarrow} g$ as in Definition \ref{defn:notion_of_convergence}. By $ L^q $-convergence of $ (g_n)_{n\geq 0} $ to $ g $, there exists a subsequence of $ (g_n)_{n\geq 0} $, denoted by $ (g_{n_i})_{i\geq0 } $ converging pointwise $ \mathcal{L}^d $-almost everywhere -- and thus $ |\Wrt ^a u | $-a.e.~-- to $ g $. Now by convergence of $ (g_{n_i})_{i\geq 0 } $ to $ Tg $ in $ L^1(\Omega,\mathbb{R}^d;\vDu) $ there exists a subsequence, again denoted by $ (g_{n_i})_{i\geq 0} $, converging to $ Tg $ $ \vDu $-a.e. Since we can write $ \vDu = | \Wrt ^a u | + | \Wrt ^s u| $ where $ |\Wrt ^s u| $ denotes the singular part of $ \vDu $ with respect to $ \mathcal{L}^d $, this implies convergence of $ (g_{n_i})_{i\geq 0} $ to $ Tg $ $ |\Wrt ^a u | $ -a.e. Together, by uniqueness of the pointwise limit, it follows $ Tg = g $ $ |\Wrt ^a u | $-a.e.

Since \[ |Tg| = |\lim _{i\rightarrow \infty } g_{n_i}|  \leq \| g\| _\infty \quad \vDu\mbox{-a.e.},\] also the second assertion follows.
\end{proof}
\end{prop}

\begin{prop}\label{prop:basic_prop_full_trace_2}For $ u\in \BV (\Omega ) $ and for any $ \phi \in C(\overline{\Omega},\mathbb{R}^d ) \cap W^q  (\dive ; \Omega ) $, it follows that $ \phi \in D $ and \[ T\phi=\phi \] as a function in $ L^1 (\Omega,\mathbb{R}^d ; \vDu ) $. If, in addition, $ u \in L^p (\Omega ) $ with $ p = \frac{q}{q-1} $ for $ 1<q<\infty $ and $ p=\infty $ for $ q=1 $ such that the normal trace operator, mapping to $ L^1(\Omega ;\vDu) $, is defined on $ D $, then for any $ g\in D $ we have that \[ T_N g = Tg \cdot \sigma _u .\] 
\begin{proof}
For the first assertion, we need to show that for any $ (\phi_n)_{n\geq 0} \overset{\sim}{\rightarrow} \phi $,
\[\intop _\Omega \vert \phi _n - \phi \vert \wrt \vDu \rightarrow 0 \mbox{ as } n \rightarrow \infty. \]
But this follows from Lebesgue's dominated convergence theorem, using that $\vert \phi _n - \phi \vert \leq 2\Vert \phi \Vert _\infty $ and that for continuous functions every point is a Lebesgue point.
Now take $ g\in D $ and assume $ u\in L^p (\Omega )  $. Since $ D\subset L^\infty (\Omega,\mathbb{R}^d) $, the normal trace $ T_N g $ is defined and, with $ (g_n)_{n\geq 0} $ as in Proposition \ref{prop:w_div_approximation}, we have
\begin{equation*}
\intop _\Omega \vert Tg\cdot \sigma _u -T_N g_n \vert \wrt \vDu \leq \intop _\Omega \vert Tg - g_n \vert \wrt \vDu \rightarrow 0.
\end{equation*}
where we used that, by Proposition \ref{prop:normal_trace_operator}, $ T_N g_n = g_n \cdot \sigma _u $ and that $ \vert \sigma _u \vert = 1 $. By closedness of $ T_N $ the second assertion follows. 
\end{proof}
\end{prop}
Note that, by density of $ C(\overline{\Omega},\R^d) $ in $ W^q(\dive;\Omega) $, Proposition \ref{prop:basic_prop_full_trace_2} in particular implies that the full trace operator is densely defined.

In \cite[Theorem 1.9]{Anzellotti83} it was shown that, for $ u\in \BV (\Omega ) \cap L^p (\Omega )$ and $ g \in W^q (\dive ;\Omega ) \cap L^\infty (\Omega,\mathbb{R}^d ) $, with $ p = \frac{q}{q-1} $ for $ 1<q<\infty $ and $ p=\infty $ for $ q=1 $, denoting by $ \theta(g,\Du) $ the normal trace of $ g $ as in Proposition \ref{prop:normal_trace_operator}, the following Gauss-Green formula holds:
\[
\intop _\Omega u \dive g \wrt x+ \intop _\Omega \theta (g,\Du) \, \vDu = \intop _{\partial \Omega } [g \cdot \nu ] u^\Omega \wrt \mathcal{H} ^{d-1},
\]
where $ [g\cdot \nu] \in L^\infty (\partial \Omega ;\mathcal{H}^{d-1})$ and $ u^\Omega \in L^1 (\partial \Omega; \mathcal{H}^{d-1} )$ denote the boundary trace functions of $ g $ and $ u $, respectively.
As an immediate consequence of this and Proposition \ref{prop:basic_prop_full_trace_2}, we can present a Gauss-Green formula for the full trace:
\begin{cor} \label{cor:full_trace_gauss_green}
For $ g\in D $, $ u\in \BV (\Omega ) \cap L^p (\Omega ) $ and $[g \cdot \nu]$ as in \cite[Theorem 1.2]{Anzellotti83}, with $ p = \frac{q}{q-1} $ for $ 1<q<\infty $ and $ p=\infty $ for $ q=1 $, we have
\[
\intop _\Omega u \dive g \wrt x + \intop _\Omega Tg\, \Wrt u = \intop _{\partial \Omega } [g \cdot \nu ] u^\Omega \wrt \mathcal{H} ^{d-1}.
\]
\end{cor}

\subsection{Subdifferential characterization}
We will now use the notion of full trace to describe the sub\-differential of the $ \TV $ functional.
In order to do so, we first remember a well known result, which provides a characterization by using an integral equation. Note that here we define
\[\TV : L^p (\Omega ) \rightarrow \overline{\mathbb{R}}, \quad 1 < p \leq \frac{d}{d-1},\]
as
\[\TV (u) = \sup \left\{ \intop _\Omega u \dive \phi \,   \bigg \vert \, \phi \in C^\infty _c (\Omega,\mathbb{R}^d) ,\, \Vert \phi \Vert _\infty \leq 1 \right\} \]
where $ \TV $ may also attain the value $ \infty $. 
\begin{prop}[Integral characterization]
\label{prop:integral_characterization} Let $ \Omega \subset \mathbb{R}^d $ with $ d\geq 2 $, $ 1<p\leq\frac{d}{d-1} $, $ q=\frac{p}{p-1} $ and $u\in L^p (\Omega) $, $u^{*}\in L^q (\Omega).$
Then $u^{*}\in\partial\TV(u)$ if and only if 
\begin{eqnarray*}
\left\{
\begin{gathered} u\in\BV(\Omega)\mbox{ and there exists }g\in W_{0}^{q}(\dive;\Omega)\\ \mbox{  with } \Vert g\Vert_{\infty}\leq1
\mbox{ such that }u^{*}=-\dive g\mbox{ and }\\
\intop_{\Omega}\mathbf{1}\wrt\vert \Wrt u\vert=-\intop_{\Omega}u\dive g .
\end{gathered}
\right.
\end{eqnarray*}
\begin{proof}
For the sake of completeness, we elaborate on the proof: Denoting by $ C=\left\{ \dive \phi \, \vert \, \phi \in C^\infty _c (\Omega,\mathbb{R}^d ) , \, \Vert \phi \Vert _\infty \leq 1 \right\} $, we have
\[\TV (u) = \mathcal{I}^* _C (u), \]where $ \mathcal{I}_C^* $ denotes the polar of $ \mathcal{I}_C $ \cite[Definition I.4.1]{Ekeland}, and, consequently, see \cite[Example I.4.3]{Ekeland},
\[\TV^* (u^*) = \mathcal{I}^{**} _C (u^*) = \mathcal{I}_{\overline{C}} (u^*)\]
where the closure of $ C$ is taken with respect to the $ L^q $ norm. Using the equivalence \cite[Proposition I.5.1]{Ekeland}
\[ u^* \in \partial \TV (u) \quad \Leftrightarrow \quad \TV (u) + \TV ^* (u^*) = ( u,u^* )_{L^p,L^q}, \]
it therefore suffices to show that \[ \overline{C}=\left\{ \dive g \, \vert \, g\in W^q_0 (\dive ,\Omega), \, \Vert g \Vert _\infty \leq 1 \right\} =:K\] to obtain the desired assertion.
Since clearly $ C\subset K $, it is sufficient for $ \overline{C}\subset K $ to show that $ K $ is closed with respect to the $ L^q $ norm. For this purpose take $ (g_n)_{n\geq 0 } \subset W^q_0 (\dive;\Omega)  $ with $ \Vert g_n \Vert _\infty \leq 1 $ such that
\[\dive g_n \rightarrow h\mbox{ in } L^q (\Omega ) \mbox{ as } n\rightarrow \infty .\]
By boundedness of $ (g_n)_{n\geq 0} $ there exists a subsequence $ (g_{n_i})_{i\geq 0 } $ weakly converging to some $ g\in L^q (\Omega ,\mathbb{R}^d) $. Now for any $ \phi \in C^\infty_c  (\Omega) $, 
\[\intop _\Omega g\cdot \nabla \phi = \underset{i\rightarrow \infty }{\lim} \intop _\Omega g_{n_i} \cdot \nabla \phi = \underset{i\rightarrow \infty }{\lim} -\intop _\Omega \dive (g_{n_i} )\phi = -\intop _\Omega h \phi, \] from which follows that $ g \in W^q (\dive ; \Omega )  $ and $ \dive g = h $. To show that $ \Vert g \Vert _\infty \leq 1 $ and $ g \in W^q _0 (\dive ; \Omega ) $ note that the set
\[ \left\{ (f,\dive f) \, \vert \, f \in W^q _0 (\dive ;\Omega), \, \Vert f\Vert_\infty  \leq 1 \right\}\subset L^q(\Omega ,\mathbb{R}^{d+1})\]
forms a convex and closed -- and therefore weakly closed -- subset of $ L^q(\Omega,\mathbb{R}^{d+1}) $ \cite[Section I.1.2]{Ekeland}. Since the sequence $ ((g_{n_i},\dive g_{n_i}))_{i\geq 0 } $ is contained in this set and converges weakly in $ L^q(\Omega,\mathbb{R}^{d+1} ) $ to $ (g,\dive g) $, we have $ g\in W^q_0 (\dive ; \Omega ) $ and $ \Vert g \Vert _\infty \leq 1 $, hence $ \dive g \in K $. For $ K \subset \overline{C} $ it suffices to show that, for any $ g \in W^q_0 (\dive ;\Omega )$ with $ \Vert g \Vert_\infty \leq 1 $ fixed, we have for all $ v\in L^p (\Omega) $ that
\[ \intop _\Omega v \dive g \leq \TV (v) \] since this implies $ \TV^* (\dive g)  = \mathcal{I}_{\overline{C}} (\dive g ) = 0 $. Now for such a $ v\in L^p (\Omega) $ we can assume that $ v\in \BV (\Omega) $ since in the other case the inequality is trivially satisfied. Thus we can take a sequence $ (v _n )_{n\geq 0} \subset C^\infty (\overline{\Omega} )$ strictly converging to $ v $ \cite[Theorem 3.9]{Ambrosio}, for which we can also assume that $ v_n \rightarrow v $ with respect to $ \Vert  \cdot \Vert _{L^p} $. Using Remark \ref{rem:gauss_green_w0_div} it follows
\begin{eqnarray*}
\intop _\Omega v \dive g  &= &\limn \intop _\Omega v_n \dive g = \limn -\intop _\Omega \nabla v_n \cdot g \\ & \leq & \limn \intop _\Omega \vert \nabla v_n \vert \vert g \vert \leq \limn \TV (v_n) = \TV (v).\, \qedhere
\end{eqnarray*}
\end{proof}
\end{prop}
\begin{rem} \label{rem:triv_subdif_inequ}
Note that in the last part of the proof of Proposition \ref{prop:integral_characterization} we have in particular shown that for any $ g \in W^q_0 (\dive ;\Omega )$ with $ \Vert g \Vert_\infty \leq 1 $, where $ q=\frac{p}{p-1} $ and $ 1<p\leq\frac{d}{d-1} $, and any $ v\in L^p (\Omega) $, the inequality
\[ \intop _\Omega v \dive g \leq \TV (v) \]
holds.
\end{rem}

Using Proposition \ref{prop:integral_characterization}, we can derive the main result of the paper, a characterization of the subdifferential of the $ \TV $ functional in terms of the full trace operator.
\begin{thm}[Pointwise characterization]
\label{thm:pointwise_characterization} With the assumptions of Proposition \ref{prop:integral_characterization}
we have that $u^{*}\in\subdif\TV(u)$ if and only if 
\begin{eqnarray*}
\left\{
\begin{gathered}
u\in\BV (\Omega) \mbox{ 
and there exists } g\in W_{0}^{q}(\dive;\Omega)\\
\mbox{  with } \Vert g\Vert_{\infty}\leq1 \mbox{ such that } u^{*}=-\dive g \mbox{ and }\\
 Tg = \sigma _u \mbox{ in } L^{1}(\Omega,\mathbb{R}^{d}; \vDu ) ,
\end{gathered}
\right.
\end{eqnarray*}
where $\sigma _u$ is the density of $\Du$ w.r.t. $\vDu$.
\end{thm}
\begin{proof}
Let $u^{*}\in\partial \TV(u)$: Using Proposition \ref{prop:integral_characterization}, with $ g\in W^q _0 (\dive ,\Omega) $ provided there, it suffices to show that, for $ (g_n)_{n\geq 0}  \subset C^\infty (\overline{\Omega},\mathbb{R}^d)$ such that $ g_n \overset{\sim}{\rightarrow} g $ it follows 
\begin{equation*}
\Vert \sigma _u - g_n \Vert _{L^1 (\Omega,\mathbb{R}^d ; \vDu )} \rightarrow 0.
\end{equation*}
Testing the zero extension of $ u $, denoted by $ w\in \BV (\mathbb{R}^d) $, with $ (g_n)_{n\geq0} $ extended to be in $ C^1 (\mathbb{R}^d,\mathbb{R}^d) $ yields, by virtue of \cite[Corollary 3.89]{Ambrosio},
\begin{eqnarray}
\intop_{\Omega}\mathbf{1}\wrt \vDu & = &- \intop_{\Omega}u\dive g \wrt x =  \underset{n\rightarrow\infty}{\lim}-\intop_{\Omega}u\dive g_{n} \wrt x \nonumber \\
 & = & \underset{n\rightarrow\infty}{\lim}-\intop_{\mathbb{R}^d}w\dive g_{n} \wrt x = \limn \intop _{\mathbb{R}^d} g_n \wrt \Wrt w \nonumber \\
 & = & \underset{n\rightarrow\infty}{\lim}\left(\intop_{\Omega}g_{n} \cdot \sigma _u\wrt \vDu +\intop_{\partial\Omega}(g_{n}\cdot \nu_{\Omega})u^{\Omega}\wrt\mathcal{H}^{d-1}\right) \label{eq:main_estimation_subdif_char}
 \end{eqnarray}
where, $u^{\Omega}\in L^{1}(\partial\Omega;\mathcal{H}^{d-1})$
denotes the trace of $u$ on $ \partial \Omega $ and $\nu_{\Omega}$ is the generalized inner
unit normal vector of $\partial\Omega$. Next, we like to show that the boundary term vanishes as $ n\rightarrow \infty $. By density of $C^{\infty}(\overline{\Omega})$ in
$\BV(\Omega)$ and continuity of the trace operator for $\BV$ functions with respect to strict convergence (see  \cite[Theorem 3.88]{Ambrosio}),
for arbitrary $\epsilon>0$, there exists $\phi_{\epsilon}\in C^{\infty}(\overline{\Omega})$
such that $\Vert u^{\Omega}-\phi_{\epsilon}^{\Omega}\Vert_{L^{1}(\partial\Omega)}<\epsilon$.
By the standard Gauss-Green theorem we can write\[
\intop_{\partial\Omega}(g_{n}\cdot\nu_{\Omega})\phi_{\epsilon}\wrt\mathcal{H}^{d-1}=-\intop_{\Omega}\dive (g_{n})\phi_{\epsilon} \wrt x-\intop_{\Omega} g_{n}\cdot\nabla\phi_{\epsilon} \wrt x\]
and taking the limit as $n\rightarrow\infty$ we get, by $ g_n \rightarrow g $ in $ W^q (\dive; \Omega ) $,\begin{eqnarray*}
\underset{n\rightarrow\infty}{\lim}\intop_{\partial\Omega}(g_{n}\cdot\nu_{\Omega})\phi_{\epsilon}\wrt\mathcal{H}^{d-1} & = & \underset{n\rightarrow\infty}{\lim}\left(-\intop_{\Omega}\dive (g_{n})\phi_{\epsilon} \wrt x-\intop_{\Omega}g_{n}\cdot\nabla\phi_{\epsilon} \wrt x\right)\\
 & = &- \intop_{\Omega}\dive (g)\phi_{\epsilon} \wrt x-\intop_{\Omega}g\cdot\nabla\phi_{\epsilon} \wrt x=0.\end{eqnarray*}
For $n\in\mathbb{N}$ we thus have, since $ \Vert g_n \Vert _\infty \leq \Vert g \Vert _\infty $,

\begin{eqnarray*}
\left|\,\intop_{\partial\Omega}(g_{n}\cdot\nu_{\Omega})u^{\Omega}\wrt\mathcal{H}^{d-1}\right| & = & \left|\,\intop_{\partial\Omega}(g_{n}\cdot\nu_{\Omega})(u^{\Omega}-\phi_{\epsilon})+(g_{n}\cdot\nu_{\Omega})\phi_{\epsilon}\wrt\mathcal{H}^{d-1}\right|\\
 & \leq & \Vert g_{n}\Vert_{\infty}\Vert u^{\Omega}-\phi_{\epsilon}\Vert_{L^{1}(\partial\Omega)}+\left|\,\intop_{\partial\Omega}(g_{n}\cdot\nu_{\Omega})\phi_{\epsilon}\wrt\mathcal{H}^{d-1}\right|\\
 & \leq & \epsilon+\left|\,\intop_{\partial\Omega}(g_{n}\cdot\nu_{\Omega})\phi_{\epsilon}\wrt\mathcal{H}^{d-1}\right|.\end{eqnarray*}
Hence \[
\underset{n}{\limsup}\left|\,\intop_{\partial\Omega}(g_{n}\cdot\nu_{\Omega})u^{\Omega}\wrt\mathcal{H}^{d-1}\right|\leq\epsilon\]
and, since $\epsilon$ was chosen arbitrarily, \[
\underset{n\rightarrow\infty}{\lim}\intop_{\partial\Omega}(g_{n}\cdot\nu_{\Omega})u^{\Omega}\wrt\mathcal{H}^{d-1}=0.\]
 Together with equation \eqref{eq:main_estimation_subdif_char} this
implies\[
\intop_{\Omega}\mathbf{1}\wrt \vDu=\underset{n\rightarrow\infty}{\lim}\intop_{\Omega}g_{n}\cdot\sigma _u\wrt \vDu.\]
Using that $\vert g_{n}(x)\vert\leq1$ for all $x\in\Omega$
and $\vert\sigma _u(x)\vert=1$, $\vDu-$a.e., we estimate $1-(g_{n}\cdot \sigma _u):$\begin{eqnarray*}
1-(g_{n}\cdot\sigma _u) & = & \frac{1}{2}\vert\sigma _u\vert^{2}-(g_{n}\cdot\sigma _u)+\frac{1}{2}\vert g_{n}\vert^{2}+\frac{1}{2}\vert\sigma _u\vert^{2}-\frac{1}{2}\vert g_{n}\vert^{2}\\
 & = & \frac{1}{2}\vert\sigma _u-g_{n}\vert^{2}+\frac{1}{2}\vert\sigma _u\vert^{2}-\frac{1}{2}\vert g_{n}\vert^{2}\\
 & \geq & \frac{1}{2}\vert\sigma _u-g_{n}\vert^{2}\quad \vDu-\mbox{a.e.}\end{eqnarray*}
Hence we have, by the Cauchy-Schwarz inequality,\begin{eqnarray*}
\underset{n\rightarrow\infty}{\limsup}\intop_{\Omega}\vert \sigma _u-g_{n}\vert\wrt \vDu & \leq & \left(\vDu (\Omega)\underset{n\rightarrow\infty}{\lim}\intop_{\Omega}\vert\sigma _u-g_{n}\vert^{2}\wrt \vDu \right)^{\frac{1}{2}}\\
 & \leq & \left(2 \vDu (\Omega)\underset{n\rightarrow\infty}{\lim}\intop_{\Omega}1-(g_{n}\cdot\sigma _u)\wrt \vDu \right)^{\frac{1}{2}}=0\end{eqnarray*}
from which the assertion follows.

In order to show the converse implication, we assume now that $u\in\BV(\Omega)$
and that there exists $g\in W_{0}^{q}(\dive;\Omega )$ with $\Vert g\Vert_{L^{\infty}}\leq1$
such that $u^{*}=-\dive g$ and $\sigma _u = Tg$ in $L^{1}(\Omega,\mathbb{R}^{d};\vDu)$.
Using Proposition \ref{prop:integral_characterization}, it is sufficient
to show that\[
\intop_{\Omega}\mathbf{1}\wrt\vDu=-\intop_{\Omega}u\dive g \wrt x.\]
Taking $(g_{n})_{n\geq0}\subset C^{\infty}(\overline{\Omega},\mathbb{R}^d)$ the
approximating sequence as in Proposition \ref{prop:w_div_approximation}, we have, analogously to the above, that
\[ \intop _{\partial \Omega } (g_n \cdot \nu _\Omega ) u^\Omega \wrt \mathcal{H}^{d-1} \rightarrow 0\] as $ n\rightarrow \infty  $ and, consequently, as $ \lim _{n\rightarrow \infty} g_n = \sigma _u $ in $ L^1(\Omega, \R ^d;\vDu) $,
\begin{eqnarray*}
\intop_{\Omega}\mathbf{1}\wrt\vDu & = & \intop_{\Omega}(\sigma _u\cdot\sigma _u)\wrt\vDu\\
 & = & \underset{n\rightarrow\infty}{\lim}\intop_{\Omega}(g_{n}\cdot\sigma _u)\wrt\vDu\\
 & = & \underset{n\rightarrow\infty}{\lim}\left(-\intop_{\Omega}\dive (g_{n})u \wrt x-\intop_{\partial\Omega}(g_{n}\cdot\nu_{\Omega})u^{\Omega}\wrt\mathcal{H}^{d-1}\right)\\
 & = & -\intop_{\Omega}\dive (g)u \wrt x.  \qedhere \end{eqnarray*}
\end{proof}
\begin{rem}
As one can see, the first two assumptions on the convergence as in Definition \ref{defn:notion_of_convergence} indeed are necessary for the techniques applied in the proof of Theorem \ref{thm:pointwise_characterization}, while the third assumption is only needed to ensure the trace operator to be the identity for continuous vector fields as in Proposition \ref{prop:basic_prop_full_trace_2}. 
\end{rem}
\begin{rem}\label{rem:ex_full_trace_condition}
Note that in the proof of Theorem \ref{thm:pointwise_characterization} we have in particular shown the following condition for existence of a trace of a $ W^q(\dive;\Omega) $ function $g$, with $ \|g \|_\infty \leq 1 $, in $ L^1(\Omega,\R ^d;\vDu) $, $ u\in L^p(\Omega)$, $ q=\frac{p}{p-1} $, $ 1<p\leq\frac{d}{d-1} $: 
\[ -\intop _\Omega u \dive g = \TV (u) \Leftrightarrow u \in \BV (\Omega), g \in D \mbox{ and } Tg = \sigma _u, \]
where $ D $ is the domain of the full trace operator $ T $ and $ \sigma _u $ is the density of $\Du$ w.r.t. $\vDu$.
\end{rem}
For the normal trace, a similar well known result follows as a direct consequence of Theorem \ref{thm:pointwise_characterization} and Proposition \ref{prop:basic_prop_full_trace_2}:
\begin{cor}
Let the assumptions of Proposition \ref{prop:integral_characterization} be satisfied. For $u\in L^p (\Omega)$ and $u^{*}\in L^q (\Omega)$
we have that $u^{*}\in\subdif\TV(u)$ if and only if 
\begin{eqnarray*}
\left\{
\begin{gathered}
u\in\BV (\Omega) \mbox{ and there exists } g\in W_{0}^{q}(\dive;\Omega) \\
\mbox{ with } \Vert g\Vert_{\infty}\leq1 \mbox{ such that }u^{*}=-\dive g \mbox{ and } \\
T_N g = Tg \cdot \sigma _u = 1 \mbox{ in } L^{1}(\Omega;\vDu).
\end{gathered}
\right.
\end{eqnarray*}
\end{cor}
At last, let us further specify the expression $ Tg = \sigma _u $. This can be done using the decomposition of $ \Du $ into an absolute continuous part with respect to the Lebesgue measure, a Cantor part and a jump part, denoted by $ \Wrt ^a u$, $\Wrt ^c u$ and $ \Wrt ^j u $, respectively \cite[Section 3.9]{Ambrosio}. The absolute continuous part can further be written as $\Wrt ^a u  = \nabla u \wrt \mathcal{L}^2 $ and the jump part as \[  \Wrt ^j u  = (u^+ (x) - u ^- (x) ) \nu _u \wrt \mathcal{H}^1 |_{S_u}\]
where $ (u^+ (x) , u^- (x) , \nu _u (x) ) $ represents uniquely, up to a change of sign, the jump at $ x\in J_u $, with $ J_u $ and $ S_u $ denoting the jump set and the discontinuity set, respectively (see \cite[Definition 3.67]{Ambrosio}). Since the measures $ \Wrt ^a u$, $\Wrt ^c u$ and $ \Wrt ^j u $ are mutually singular and $ \mathcal{H}^1 (S_u \setminus J_u) = 0 $, the following result follows from Theorem \ref{thm:pointwise_characterization} and Proposition \ref{prop:basic_prop_full_trace_1}.

\begin{prop} \label{prop:spec_trace}
Let the assumptions of Proposition \ref{prop:integral_characterization} be satisfied. For $u\in L^p (\Omega)$ and $u^{*}\in L^q (\Omega)$
we have that $u^{*}\in\subdif\TV(u)$ if and only if $u\in\BV (\Omega)$ and there exists $g\in W_{0}^{q}(\dive;\Omega)$ with $\Vert g\Vert_{\infty}\leq1$  such that $u^{*}=-\dive g$ and 
\begin{align*}
 g = \frac{\nabla u}{|\nabla u|} \quad & \mathcal{L}^d -a.e. \mbox{ on } \Omega \setminus \{x:\nabla u (x) = 0\}, \\
 Tg  = \frac{u^+ (x) - u ^- (x)}{|(u^+ (x) - u ^- (x) )|} \nu _u \quad  & \mathcal{H}^1 -a.e. \mbox{ on } S_u,\\
Tg  = \sigma _{C_u} \quad  & |\Wrt ^c u| -a.e.,
\end{align*}
where $ \sigma _{C_u} $ is the density function of $ \Wrt ^c u$ with respect to $ | \Wrt ^c u | $.
\end{prop}

\section{Applications}
In this section we will present some applications where the notation of a full trace together with the subdifferential characterization of the previous section can be used to extend known results involving the subdifferential of the $ \TV $ functional. Remember that $ \Omega $ is always assumed to be a bounded Lipschitz domain. For simplicity, we now restrict ourselves to the two dimensional setting, i.e. $ \Omega \subset \R ^2 $, and use the more common notation $ H(\dive;\Omega) $ for the space $ W^2(\dive ; \Omega) $. 

As already mentioned in the introduction, the term of normal trace for $ H(\dive;\Omega) $ functions is frequently used to describe the total variational flow, i.e. the solution of the formal equation \cite{Andreu01dirichlet,Andreu01}
\[ (\mathcal{P}_F) \begin{cases} \frac{\partial u}{\partial t} = \dive \left( \frac{\Du}{\vDu} \right) & \mbox{in}  \quad (0,\infty ) \times \Omega  \\  u(0,\cdot ) = u_0(\cdot ) & \mbox{in} \quad  \Omega . \end{cases} \]
Defining the functional $ \TV :L^2 (\Omega ) \rightarrow \overline{\mathbb{R}} $, this corresponds to the evolution problem
\[ (\mathcal{P}) \begin{cases} \frac{\partial u(t)}{\partial t} +\partial \TV (u(t)) \ni 0 & \mbox{for}  \quad t \in (0,\infty )  \\  u(0) = u_0 & \mbox{in} \quad  L^2(\Omega) \end{cases} \]
which appears in the steepest descent method to minimize the $ \TV $ functional.

A solution to $ (\mathcal{P}) $ is a continuous function $ u:[0,\infty) \rightarrow  L^2(\Omega)$ with $ u(0) = u_0 $, which is absolutely continuous on $ [a,b] $ for each $ 0<a<b $, and hence differentiable almost everywhere, with $ \frac{\partial u}{\partial t}\in L^1 ((a,b),L^2(\Omega)) $ 
and $ -\frac{\partial u(t)}{\partial t}\in \partial \TV (u(t)) $ for almost every $ t\in (0,\infty) $.

Using this notation, one gets the following existence result:

\begin{prop}
Let $u_0 \in L^2(\Omega)$. Then there exists a unique solution 
to $(\mathcal{P})$.
\begin{proof}
  Using \cite[Corollary I.6.2]{Ekeland} it follows that the closure of the domain of $\partial \TV$ is already $ L^2(\Omega) $ and thus the result follows from \cite[Corollary IV.3.2]{ShowalterMOP}
\end{proof}
\end{prop}

Using the full trace operator $T$ and Theorem 
\ref{thm:pointwise_characterization} we can now provide an equivalent 
characterization of a solution to $(\mathcal{P})$. For the proof, we
need some properties for the solution which are stated in a lemma.

\begin{lem}
  \label{lem:prop_tv_flow}
  Consider $\partial \TV$ as a maximal monotone operator on 
  $L^2(\Omega)$ and denote by
  \[
  A_0(u) = \argmin_{v \in \partial \TV(u)} \ \|v\|_{L^2}
  \]
  the \emph{minimal section} of $\partial \TV$.
  
  If $u_0 \in \domain (\partial \TV)$, then the
  solution $u$ of ($\mathcal{P}$) satisfies:
  \begin{enumerate}[(i)]
  \item
    \label{item:tv_flow_1}
    $u: [0,\infty) \rightarrow L^2(\Omega)$ is right-differentiable
    with right-derivative $D^+ u$ solving
    \[
    D^+u(t) + A_0\bigl( u(t) \bigr) = 0 \qquad
    \text{for all} \ t \geq 0,
    \]
  \item
    \label{item:tv_flow_2}
    $A_0 \circ u: [0,\infty) \rightarrow L^2(\Omega)$,
    $(A_0 \circ u)(t) = A_0\bigl( u(t) \bigr)$ is right-continuous with
    $t \mapsto \| A_0\bigl( u(t) \bigr) \|_{L^2}$ non-increasing,
  \end{enumerate}
\end{lem}

\begin{proof}
  The items \ref{item:tv_flow_1} and \ref{item:tv_flow_2} follow directly
  from \cite[Proposition IV.3.1]{ShowalterMOP} applied to $\partial \TV$.
\end{proof}

The characterization of the total variation flow in terms of the
full trace then reads as follows.

\begin{prop}
A continuous function $ u:[0,\infty) \rightarrow  L^2(\Omega)$ is a 
solution to $ (\mathcal{P}) $ if and only if 
\begin{enumerate}[(i)]
\item \label{item:tv_flow_char_1}$u$ is  absolutely continuous 
  on $ [a,b] $ for each $ 0<a<b $ with derivative
  $ \frac{\partial u}{\partial t}\in L^1 ((a,b);L^2(\Omega)) $,
\item \label{item:tv_flow_char_2}
  $ u(t) \in \BV(\Omega) $ for each $t > 0$, $ u(0) = u_0 $,
\item \label{item:tv_flow_char_3}
  there exists
  $ g\in L^\infty ( (0,\infty)\times \Omega, \R^d) $ with 
  $ \| g\| _\infty \leq 1 $ and
\item \label{item:tv_flow_char_4}
  $ g: (0,\infty) \rightarrow H_0(\dive;\Omega)$ is measurable
  with
  $ \frac{\partial u(t)}{\partial t} = \dive g(t) $ as well as
  \[
  Tg(t) = \sigma _u (t) \quad \mbox{in} 
  \quad L^1(\Omega, \mathbb{R}^2;|\Du (t) |) 
  \]
  for almost every $ t\in (0,\infty) $.
\end{enumerate}
\end{prop}

\begin{proof}
  First note that without loss of generality, we can assume that
  $u_0 \in \domain(\partial \TV)$: From
  \cite[Proposition IV.3.2]{ShowalterMOP} follows that for each 
  $t_0 > 0$, the translated solution
  $t \mapsto u(t + t_0)$ solves ($\mathcal{P}$) with initial
  value $u(t_0) \in \domain(\partial \TV)$. Consequently,
  if the claimed statements are true on each $[t_0,\infty)$, then also 
  on $(0,\infty)$.

  Choose $L > 0$. We will now approximate $u$ on $[0,L)$ 
  as well as $\frac{\partial u}{\partial t}$
  by piecewise constant functions as follows. 
  Denote by $0 = t_0 < t_1 < \ldots < t_K = L$ a partition 
  of $[0,L)$. For $t \in [0,L)$ denote by $k(t) = \min
  \ \{k': t_{k'} > t\}$ as well as $\tau(t) = t_{k(t)} - t_{k(t)-1}$. 
  For each $\varepsilon > 0$ we can
  now choose, due to the uniform continuity of $u$ on $[0,L]$, 
  a partition which satisfies
  \[
  \|u(t) - u(t_{k(t)})\|_{L^2} < \varepsilon.
  \]
  for all $t \in [0,L)$. It is moreover possible to achieve that these
  partitions are nested which implies that $t_{k(t)} \rightarrow t$,
  $\tau(t) \rightarrow 0$ as $\varepsilon \rightarrow 0$, both
  monotonically decreasing. Then, the function
  \[
  u^\varepsilon: [0,L) \rightarrow L^2(\Omega), 
  \qquad u^\varepsilon(t) = u(t_{k(t)})
  \]
  obviously converges to $u$ in $L^\infty((0,L), L^2(\Omega))$.
  Likewise, the function
  \[
  (u^\varepsilon)': [0,L) \rightarrow L^2(\Omega), \qquad
  (u^\varepsilon)'(t) = -A_0\bigl( u(t_{k(t)})\bigr)
  \]
  satisfies, on the one hand,
  $-(u^\varepsilon)'(t) \in \partial \TV\bigl( u^\varepsilon(t) \bigr)$
  for $t \in [0,L)$ by definition of $A_0$, see 
  Lemma~\ref{lem:prop_tv_flow}. 
  On the other hand, for $t \in [0,L)$, we have
  $t_{k(t)} \rightarrow t$ monotonically
  decreasing, which implies by the right continuity of $t \mapsto 
  A_0\bigl( u(t) \bigr)$, see Lemma~\ref{lem:prop_tv_flow}, that
  \[
  \lim_{\varepsilon \rightarrow 0} (u^\varepsilon)'(t) 
    = -A_0\bigl( u(t) \bigr) \qquad \text{in} \ L^2(\Omega).
  \]
  Also $\|(u^\varepsilon)'(t)\|_2 \leq \|A_0(u_0)\|_2$, again by 
  Lemma~\ref{lem:prop_tv_flow}, so there exists an integrable majorant
  and by Lebesgue's theorem, $\lim_{\varepsilon \rightarrow 0} (u^\varepsilon)' 
  = -A_0 \circ u$ in $L^2((0,L), L^2(\Omega))$. 
  However, Lemma~\ref{lem:prop_tv_flow} yields $-A_0 \circ u = D^+ u$,
  so $(u^\varepsilon)'$ is indeed approximating 
  $\frac{\partial u}{\partial t}$.

  As each $u^\varepsilon$, $(u^\varepsilon)'$ is constant 
  on the finitely
  many intervals $[t_{k(t)-1}, t_{k(t)})$ and 
  $-(u^\varepsilon)'(t) \in \partial \TV\bigl(
  u^\varepsilon(t) \bigr)$, we can 
  choose a vector field $g$ according 
  to Proposition~\ref{prop:integral_characterization}
  on each of these intervals.
  Composing these $g$ yields a measurable 
  $g^\varepsilon \in L^2((0,L); H_0(\dive, \Omega))$,
  $\|g^\varepsilon\|_\infty \leq 1$ in $L^\infty((0,L) \times \Omega, \R^d)$ 
  and such that $(u^\varepsilon)' 
  = \dive g^\varepsilon$ in the weak sense. Moreover,
  \begin{equation}
    \label{eq:flow_approx}
    \int_0^L \int_\Omega \ones \wrt{|\Wrt u^\varepsilon(t)|} \wrt t= 
    - \int_0^L \int_\Omega u^\varepsilon \dive g^\varepsilon \wrt x \wrt t.
  \end{equation}
  Now, $\{g^\varepsilon\}$ is bounded in $L^2((0,L), H_0(\dive, \Omega))$,
  hence there exists a weakly convergent subsequence (not relabeled) and
  a limit $g$ with $\|g\|_\infty \leq 1$ in $L^\infty((0,L) \times \Omega, \R^d)$. 
  In particular, as $(u^\varepsilon)' = \dive g^\varepsilon$,
  we have $\dive g^\varepsilon \rightarrow \frac{\partial u}{\partial t}$
  in $L^2((0,L), L^2(\Omega))$. By weak closedness of the divergence operator,
  also $\dive g = \frac{\partial u}{\partial t}$.

  Finally, taking the limits in~\eqref{eq:flow_approx} yields
  \[
  \int_0^L \int_\Omega \ones \wrt|\Wrt u| \wrt t \leq 
  \liminf_{\varepsilon \rightarrow 0} \int_0^L \int_\Omega
  \ones \wrt|\Wrt u^\varepsilon| \wrt t = 
  - \int_0^L \int_\Omega u \dive g \wrt x \wrt t. 
  \]
  On the other hand, as for almost every $t \in (0,L)$, 
  $g \in H_0(\dive; \Omega)$ and $\|g(t)\|_\infty \leq 1$, according to Remark \ref{rem:triv_subdif_inequ}
  it follows that $-\int_\Omega u(t) \dive g(t) \leq \TV\bigl(u(t)\bigr)$.
  Hence, the above is only possible if $- \int_\Omega u(t) \dive g(t) 
  = \TV\bigl(u(t) \bigr)$ for almost every $t \in (0,L)$.
  By Remark \ref{rem:ex_full_trace_condition}, a full trace then
  exists, i.e.
  \[
  Tg(t) = \sigma_{u}(t) \quad \text{in} \quad L^1(\Omega, \R^d; 
  |\Wrt u(t)|) \qquad \text{for a.e.} \ t \in (0,L).
  \]
  Conversely, if we now assume that $ u:[0,\infty )\rightarrow L^2(\Omega) $ satisfies \ref{item:tv_flow_char_1} - \ref{item:tv_flow_char_4}, in order to establish that $ u $ is a solution to $ (\mathcal{P}) $ it is left to show that $ -\frac{\partial u(t)}{\partial t} \in \subdif \TV (u(t)) $ for almost every $ t\in (0,\infty) $. But since at almost every $ t\in (0,\infty) $ we have, for $g\in L^\infty ( (0,\infty)\times \Omega, \R^d) $ as in \ref{item:tv_flow_char_3}, that $ g(t)\in H_0(\dive;\Omega) $, $ \|g(t) \|_\infty\leq 1  $, $\frac{\partial u(t)}{\partial t} = \dive g(t)$ and $ Tg(t) = \sigma _u (t) $,  this follows as immediate consequence of Theorem~\ref{thm:pointwise_characterization}.
\end{proof}

In a related context, a Cheeger set \cite{Cheeger70,Kawohl06} of a bounded set $ G $ of finite perimeter \cite[Section 3.3]{Ambrosio} is defined to be the minimizer of
\begin{equation} \underset{A\subset \overline{G}}{\min} \frac{|\partial A|}{|A|} .\end{equation}
Defining the constant
\[\lambda _G = \frac{|\partial G |}{| G |} ,\]
a sufficient condition for $ G $ to be a Cheeger set of itself, or in other words to be calibrable, is that $ v:= \chi _G $  satisfies the equation \cite[Lemma 3]{Bellettini02}
\begin{equation} \label{eq:calibrable} - \dive (\sigma _v ) = \lambda _G v \quad \mbox{ on } \mathbb{R}^2,\end{equation}
i.e. there exists a vector field $ \xi \in L^\infty (\mathbb{R}^2; \mathbb{R}^2) $ such that $ \| \xi \| _\infty \leq 1 $,
\[-\dive \xi = \lambda _G v \quad \mbox{ on } \Rtwo \]
and
\[ \intop _{\mathbb{R}^2} \theta(\xi, \Wrt v)\wrt |\Wrt v| = \intop _{\mathbb{R}^2} \mathbf{1}\wrt |\Wrt v| .\]
This condition is further equivalent to \cite[Theorem 4]{Bellettini02}:
\begin{enumerate}
\item $G $ is convex.
\item $\partial G$ is of class $ C^{1,1} $.
\item It holds
\[\underset{p}{\mbox{ess}\sup}\,\kappa _{\partial G} (p) \leq \frac{P(G)}{|G|}, \]
\end{enumerate}
where $ \kappa _{\partial G} $ is the curvature of $ \partial G $.
Using the full trace operator, we can provide the following sufficient condition for $ G $ being calibrable:
\begin{prop}
Let $ G \subset \mathbb{R}^2 $ be a bounded set of finite perimeter. Then $ v= \chi _G \in \BV (\Rtwo) $ satisfies condition \eqref{eq:calibrable} if there exists a bounded Lipschitz domain $ K $ such that $ \overline{G}\subset K $ and $ \xi \in H_0  (\dive;K) $ with $ \| \xi \| _\infty \leq 1 $ and $ \xi \in D $, where $ D $ is the domain of the full trace operator, such that
\[ -\dive \xi = \lambda _G v \quad \mbox{ on }K\]
and 
\[ T\xi = \nu _G \quad  \mathcal{H}^1 - \mbox{ almost everywhere on } \mathcal{F} G ,\]
where $ \mathcal{F}G $ is the reduced boundary, i.e. the set of all points $ x\in \supp |\Wrt \chi _G | $ such that the limit
\[ \nu _G (x) := \underset{\rho \rightarrow 0^+}{\lim}  \frac{\Wrt \chi _G (B_\rho (x) )}{|\Wrt \chi _G (B_\rho (x)) |}\]
exists.
\end{prop}
\begin{proof}
The proof is straightforward: Using that $ |\Wrt \chi _G | = \mathcal{H}^1|_{\mathcal{F}G } $ and that $ \Wrt \chi _G = \nu _G |\Wrt \chi _G | $ \cite[Section 3.5]{Ambrosio} it follows that 
\[\intop _K |\Wrt v| = \intop _K T\xi \cdot \nu _G  \wrt |\Wrt v| = \intop _K \theta (\xi,\Wrt v)\wrt |\Wrt v |. \] From this and the fact that $ \xi \in H_0  (\dive ;K) $ it follows that its extension by $ 0 $ to the whole $ \mathbb{R}^2 $ is contained in $ H(G ;\Rtwo) $ and satisfies condition \eqref{eq:calibrable}.
\end{proof}

The full trace operator can also be used to formulate optimality conditions for optimization problems appearing in mathematical imaging. A typical problem formulation would be
\begin{equation} \min\limits_{u\in L^2(\Omega)} \TV (u)+F (u), \label{eq:imaging_prob} \end{equation}
where $ \TV $ plays the role of a regularization term and $ F:L^2 (\Omega) \rightarrow \overline{\mathbb{R}} $ reflects data fidelity. Under weak assumptions on $ F $ we can derive the following general optimality condition:

\begin{prop}\label{prop:imaging_prob}
Suppose that $ F:L^2(\Omega) \rightarrow \overline{\mathbb{R}} $ is such that $ \partial (\TV + F ) = \partial \TV + \partial F $. Then we have that $ u\in L^2(\Omega) $ solves \eqref{eq:imaging_prob} if and only if there exists $ g \in H_0(\dive ;\Omega) $ such that $ \| g\| _\infty \leq 1 $,
\[\dive g \in \partial F (u) \]
and
\[Tg = \sigma _u  \quad \mbox{in}\quad L^1 (\Omega ,\mathbb{R}^2 ;\vDu )\]

\end{prop}
\begin{proof}
This follows immediately from $ \partial (\TV + F) = \partial \TV + \partial F $ and the characterization of $ \partial \TV $ in Theorem \ref{thm:pointwise_characterization}.
\end{proof}
In \cite{Vese01}, a problem of this type, but with a generalized regularization term was considered. Existence and a characterization of solutions to
\[\min _{u\in \BV}  \intop _\Omega \varphi (\vDu ) + \intop _\Omega | Ku - u_0 |^2 \] was shown, a problem which appears in denoising, deblurring or zooming of digital images.. For the characterization of optimal solutions, again the term $ g \cdot \sigma _u $, with $ g\in H(\dive ;\Omega) $, was associated to a measure and then, following \cite{Demengel84}, it was split into a measure corresponding the absolute continuous part of $ \Du $ with respect to the Lebesgue measure and a singular part. By applying Propositions \ref{prop:spec_trace} and \ref{prop:imaging_prob}, we can now get a characterization of solutions similar to \cite[Propostion 4.1]{Vese01}, but in terms of $ L^1 (\Omega ,\mathbb{R}^2; \vDu) $ functions, for the special case that $ \varphi $ is the identity:
\begin{prop}
Let $ u_0 \in L^2 (\Omega) $ and $ K:L^2(\Omega ) \rightarrow L^2 (\Omega) $ a continuous, linear operator. Then, $ u\in L^2(\Omega) $ is a solution to
\[\min _{u\in \BV}  \intop _\Omega \vDu + \intop _\Omega | Ku - u_0 |^2 \]
if and only if $ u\in \BV(\Omega ) $ and there exists $ g \in H_0(\dive ;\Omega )$ with $ \|g \|_\infty \leq 1 $ such that
\[ 2K^* (Ku - u_0) = \dive g \]
and
\begin{align*}
 g = \frac{\nabla u}{|\nabla u|} \quad & \mathcal{L}^2 -a.e. \mbox{ on } \Omega \setminus \{x:\nabla u (x) = 0\} \\
 Tg  = \frac{u^+ (x) - u ^- (x)}{|(u^+ (x) - u ^- (x) )|} \nu _u \quad  & \mathcal{H}^1 -a.e. \mbox{ on } S_u\\
Tg  = \sigma _{C_u} \quad  & |\Wrt ^c u| -a.e.,
\end{align*}
where $ u^+,u^-, \nu _u , S_u, C_u, \nabla u $ and $ |\Wrt ^c u| $ are defined as in Proposition \ref{prop:spec_trace} and its preceding paragraph.
\begin{proof}
By continuity of $ F(u) = \intop _\Omega |Ku - u_0 |^2 $ it follows that $ \partial (\TV + F) = \partial \TV + \partial F $ and we can apply Proposition \ref{prop:imaging_prob}. The characterization follows then by Proposition \ref{prop:spec_trace} and the fact that $ \partial F(v) = \{ 2K^* (Ku - u_0) \} $ for any $ v\in L^2(\Omega) $.
\end{proof}
\end{prop}
The general formulation of an imaging problem as in \eqref{eq:imaging_prob} also applies, for example, to the minimization problem presented in \cite{Holler12}: There, as part of an infinite dimensional modeling of an improved JPEG reconstruction process, one solves
\begin{equation} \min\limits_{u\in L^2(\Omega)} \TV (u)+\mathcal{I}_U (u) \label{eq:reconst_prop} \end{equation}
where $ U=\{u\in L^2 (\Omega)\,\vert\,Au\in J_n \mbox{ for all } n\in \mathbb{N}\} $, $ A:L^2(\Omega)  \rightarrow \ell ^2$ is a linear basis transformation operator and $ (J_n)_{n\in \mathbb{N}} = ([l_n,r_n])_{n\in \mathbb{N}}$ a given data set. Under some additional assumptions, a necessary and sufficient condition for $ u $ being a minimizer of \eqref{eq:reconst_prop} is stated in \cite[Theorem 5]{Holler12}. Using the full trace operator, this condition can now be extended as follows:
\begin{prop} With the assumptions of \cite[Theorem 5]{Holler12}, the function $ u\in L^2(\Omega) $ is a minimizer of \eqref{eq:reconst_prop} if and only if $u\in\BV(\Omega)\cap U$ and there exists $g\in H_{0}(\dive;\Omega)$
satisfying 
\begin{enumerate}
\item $\Vert g\Vert_{\infty}\leq1$,

\item $Tg = \sigma _u, \, \vDu \mbox{-almost everywhere}$,

\item $\left\{ \begin{array}{l}
(\dive g,a_{n})_{L^{2}}\geq0\mbox{\,\ if\,}(Au)_{n}=r_{n}\neq l_{n},\\
(\dive g,a_{n})_{L^{2}}\leq0\mbox{\,\ if\,}(Au)_{n}=l_{n}\neq r_{n},\\
(\dive g,a_{n})_{L^{2}}=0\mbox{\,\ if\,}(Au)_{n}\in\overset{\circ}{J}_{n},\end{array}\right.\quad\forall n\in \mathbb{N}.$
\end{enumerate}
\end{prop}

\section{Conclusion}
We have introduced a trace operator allowing a pointwise evaluation of $ W^q (\dive ;\Omega) $ functions in the space $ L^1 (\Omega ,\mathbb{R}^d; \vDu) $, for $ u\in \BV(\Omega) $. Using this operator, we have derived a subdifferential characterization of the total variation functional when considered as a functional from $ L^p (\Omega) $ to the extended reals. This characterization gives an analytical motivation for the notation
\[ -\dive \left( \frac{\nabla u}{| \nabla u |} \right) \in \partial \TV (u), \] frequently used in mathematical imaging problems related to $ \TV $ minimization. We further have shown that, as on would expect, the concept of full trace extends the normal trace term by Anzellotti \cite{Anzellotti83} and that it can be used in several applications, for example, to characterize the total variational flow.

\begin{appendix}
\section{An approximation result}

Since existence of a suitable approximating sequence for $ W^q (\dive ;\Omega ) $-vector fields is frequently used in this work, we give here an example of how to construct such a sequence.
For $ \Omega $ a bounded Lipschitz domain, $ 1\leq q <\infty  $ and $ g\in W^q (\dive ;\Omega) $, we have to show existence of $ (g_n)_{n\geq 0 }\subset C^\infty (\overline{\Omega} ,\mathbb{R}^d) $ satisfying:

\begin{enumerate}
\item $\Vert g_n - g \Vert _{W^q(\dive)} \rightarrow 0 \quad \mathrm{ as } \quad n \rightarrow \infty $,
\item $\Vert g_n \Vert _\infty \leq \Vert g \Vert _\infty$ for each $ n\in \mathbb{N} $ if $ g\in L^\infty (\Omega,\mathbb{R}^d)\cap W^q(\dive ;\Omega ), $
\item $g_n (x) \rightarrow g(x) $ for every Lebesgue point $x\in \Omega $ of $g$,
\item $\Vert g_n - g \Vert _{\infty,\overline{\Omega}} \rightarrow 0$ as $ n\rightarrow \infty $, if, additionally, $ g\in C(\overline{\Omega},\mathbb{R}^d) $.
\end{enumerate}
\begin{proof}
The proof follows basic ideas presented in \cite[Theorem 4.2.3]{Evans} for a density proof for Sobolev functions. We make use of the Lipschitz property of $ \partial \Omega $: For $ x \in \partial \Omega$, take $ r>0 $ and $ \gamma : \mathbb{R}^{d-1} \rightarrow \mathbb{R} $ Lipschitz continuous, such that -- upon rotating and relabeling the coordinate axes if necessary -- we have
\begin{equation}\label{eq:lipschitz}
\Omega \cap Q_r (x) = \{ y \in \mathbb{R}^d \, \vert \, \gamma (y_1,\ldots,y_{d-1})<y_d \} \cap Q_r (x)
\end{equation}
where $ Q_r (x) = \{ y \in \mathbb{R}^d\, \vert \, \vert y_i - x_i \vert <r \, , \, i=1,..,d \} $. Now for fixed $ x\in \partial \Omega $, we define $ Q = Q_r (x) $ and $ Q' = Q_{\frac{r}{2}} (x) $.
In the first step, we suppose that \[ \mathrm{spt}(g) := \overline{\{y\in \Omega : g(y)\neq 0\}} \subset Q'\] and show that there exist vector fields $ g_\epsilon \in C^\infty (\overline{\Omega},\mathbb{R}^d) $ converging, as $ \epsilon \rightarrow 0 $, to $ g $ -- in $ W^q (\dive ;\Omega) $, pointwise in every Lebesgue-point $ y\in \Omega $ and uniformly on $ \overline{\Omega} $ if additionally $ g\in C(\overline{\Omega},\mathbb{R}^d) $ -- and satisfying the boundedness property 2). 

Choose  $ \alpha = \mathrm{Lip}(\gamma ) + 2 $ fixed and $ 0<\epsilon < \frac{r}{2(\alpha + 1)} $ arbitrarily. It follows then by straightforward estimations that, for any $ y\in \overline{\Omega \cap Q'} $, with $ y^\epsilon =y+\epsilon \alpha e_d$, where $ e_d $ is the $d$th coordinate vector according to \eqref{eq:lipschitz}, we have $ \overline{B_\epsilon (y^\epsilon )} \subset \Omega \cap Q $. 
Now with $ \eta : \mathbb{R}^d \rightarrow \mathbb{R} $ a standard mollifier kernel supported in the unit ball, we define
\[\eta_\epsilon (y ) = \frac{1}{\epsilon^d}\eta \left( \frac{y}{\epsilon} \right).\]
Using that $ \overline{B_\epsilon (y^\epsilon )} \subset \Omega \cap Q $, for $ y\in \overline{\Omega \cap Q'} $, it follows that  the support of the functions
\[ x \mapsto \eta_\epsilon (y + \epsilon \alpha e_d - x ) \]
is contained in $ \Omega \cap Q $. Thus, for $ 1\leq j \leq d $, the functions $ g_\epsilon ^j :\overline{\Omega \cap Q'} \rightarrow \R $,
\begin{eqnarray}
g_\epsilon ^j (y) & = &\intop _{\mathbb{R}^d} \eta_\epsilon (y + \epsilon \alpha e_d -x) g^j(x) \wrt x \\ \nonumber
& = & \intop_{\mathbb{R}^d} \eta _\epsilon (y- z) g^j (z + \epsilon \alpha e_d) \wrt z = \left( \eta_\epsilon * g^j_{S_\epsilon} \right)(y),
\end{eqnarray}where 
\[ g^j_{S_\epsilon} (y) := g^j (y + \epsilon \alpha e_d) \] denotes the composition of $ g^j $ with a translation operator, are well defined. Using standard results, given for example in \cite[Section 2.12 and Proposition 2.14]{Alt}, it follows that $ g_\epsilon ^j \in C^\infty (\overline{\Omega \cap Q'})$ and, extending by $ 0 $ outside of $ \overline{\Omega \cap Q'}  $, that
\begin{eqnarray*}
   \Vert g_\epsilon ^j - g^j \Vert _{L^q(\Omega \cap Q')} & \leq & \Vert \eta_\epsilon * g_{S_\epsilon}^j    - \eta_\epsilon * g^j \Vert_{L^q(\mathbb{R}^d)} + \Vert \eta_\epsilon * g^j - g^j \Vert _{L^q(\mathbb{R}^d)} \\
    & \leq & \Vert \eta_\epsilon \Vert _{L^1(\mathbb{R}^d)} \Vert g_{S_\epsilon}^j    - g^j \Vert_{L^q(\mathbb{R}^d)}+ \Vert \eta_\epsilon * g^j - g^j \Vert _{L^q(\mathbb{R}^d)}\rightarrow 0
\end{eqnarray*}
as $\epsilon \rightarrow 0 $. By equivalence of norms in $ \mathbb{R}^d $ it thus follows that the vector valued functions $ g_\epsilon = (g_\epsilon ^1, \ldots, g_\epsilon ^d) $ are contained in $ C^\infty (\overline{\Omega \cap Q'}) $ and that $ \Vert g_\epsilon - g \Vert _{L^q(\Omega \cap Q')} \rightarrow 0 $ as $\epsilon \rightarrow 0  $.
Since, for $ i\in \{1 \ldots d\} $, 
\[\partial _i (\eta_\epsilon * g^j_{S_\epsilon} )= \partial _i \eta _\epsilon * g^j _{S_\epsilon}, \] we have, for $ y\in \overline{\Omega \cap Q'} $, that
\begin{eqnarray*}
\dive g_\epsilon (y) & = &  \intop _{\mathbb{R}^d}  \nabla _y (\eta_\epsilon (y -x)) \cdot g_{S_\epsilon}(x) \wrt x  \\
 & = &  \intop _{\Omega \cap Q}  \nabla _y (\eta_\epsilon (y + \epsilon \alpha e_d -z)) \cdot g(z) \wrt z  \\
 & = & - \intop _{\Omega \cap Q}  \nabla _z (\eta_\epsilon (y + \epsilon \alpha e_d -z)) \cdot g(z) \wrt z  \\
 & = &   \intop _{\Omega \cap Q}  (\eta_\epsilon (y +  \epsilon \alpha e_d -z)) \dive g(z) \wrt z  \\
  & = &   \intop _{\mathbb{R}^d}  (\eta_\epsilon (y +  \epsilon \alpha e_d -z)) \dive g(z) \wrt z,  \\
\end{eqnarray*} where we used that $  x\mapsto \eta _\epsilon (y + \epsilon \alpha e_d - x)\in C_c ^\infty (\Omega \cap Q) $ and the weak definition of $ \dive $. An argumentation analogous to the above thus yields $\Vert \dive g_\epsilon - \dive g \Vert _{L^q (\Omega \cap Q')} \rightarrow 0  \mbox{ as } \epsilon \rightarrow 0 $.
Now let $ y\in \Omega \cap Q' $ be a Lebesgue point of $ g $. Again by equivalence of norms it suffices to show that $ g_\epsilon ^j (y) \rightarrow g^j (y) $ for $ y $ being a Lebesgue point of $ g^j $, $ 1\leq j \leq d $. With $ \epsilon>0 $ sufficiently small such that, with $ t:=1+\alpha $, we have $ B_{\epsilon t} (y) \subset \Omega \cap Q $ we can estimate
\begin{eqnarray*}
\vert g^j _\epsilon (y) -g^j (y) \vert & = &  \bigg \vert \frac{1}{\epsilon ^d} \intop _{\mathbb{R}^d}  \eta \Bigr(\frac{y-w}{\epsilon}\Bigl)  \left( g^j(w+\epsilon \alpha e_n) - g^j (y) \right) \wrt w \bigg \vert \\
 & \leq  & C(d) \frac{1}{\vert B_\epsilon (y) \vert } \intop _{B_\epsilon (y)}  \vert g^j(w+\epsilon \alpha e_n)-g^j(y) \vert \wrt w \\
 & = & C(d)\frac{1}{\vert B_\epsilon (y)\vert} \intop _{B_\epsilon (y+\epsilon \alpha e_n)}  \vert g^j(w)-g^j(y) \vert \wrt w \\ & \leq & \tilde{C}(d) \frac{1}{\vert B_{\epsilon t}(y) \vert} \intop _{B_{\epsilon t} (y)}  \vert g^j(w)-g^j(y) \vert \wrt w,
\end{eqnarray*}
with $ C(d),\tilde{C}(d)>0 $ constants depending only on $ d $. Now since $ y $ was assumed to be a Lebesgue point of $ g^j $, the desired convergence follows.

Now, additionally suppose that $ g\in C(\overline{\Omega},\mathbb{R}^d) $. Note that $ \epsilon >0  $ can also be chosen such that with $ \tau = \alpha + 1 $, $ B_{\epsilon t} \subset \Omega \cap Q $ for all $ y\in \overline{\Omega \cap Q'} $, so the above implies
\begin{eqnarray*} \vert g^j _\epsilon (y) -g^j (y) \vert &\leq& \tilde{C}(d) \frac{1}{\vert B_{\epsilon t}(y) \vert} \intop _{B_{\epsilon t} (y)\cap \Omega \cap Q}  \vert g^j(w)-g^j(y) \vert \wrt w \\
 &\leq &\tilde{C}(d) \underset{w\in \overline{B_{\epsilon t} (y)\cap \Omega \cap Q}}{\sup}\left(| g^j(w) - g^j(y) |\right).
 \end{eqnarray*}
By uniform continuity of $ g $ in the compact set $ \overline{\Omega} $ it follows that $ \Vert g^j _\epsilon - g^j \Vert _{\infty,\overline{\Omega \cap Q}}$ -- and thus also $ \Vert g_\epsilon - g \Vert _{\infty,\overline{\Omega \cap Q'}}  $ -- converges to zero as $ \epsilon \rightarrow 0  $.

Next we estimate the sup-norm of $ g_\epsilon $: Suppose $ \Vert g \Vert _\infty \leq C$. For $ y\in \overline{\Omega \cap Q'} $ we then have:
\begin{eqnarray*}
\vert g_\epsilon (y) \vert^2  & = & \frac{1}{\epsilon ^{2d}} \sum _{i=1} ^d \left( \, \intop  _{\Omega\cap Q}  \sqrt{\eta \Bigl(\frac{y-w}{\epsilon} + \alpha e_n \Bigr)}\sqrt{\eta \Bigl(\frac{y-w}{\epsilon} + \alpha e_n \Bigr)}g^i (w) \wrt w \right) ^2 \\ 
 & \leq & \frac{1}{\epsilon ^{2d}}  \left(\, \intop  _{\Omega\cap Q} \eta \Bigl(\frac{y-w}{\epsilon} + \alpha e_n \Bigr)  \sum _{i=1} ^d g^i (w)^2 \wrt w \right) \cdot \\
 & & \left( \,\intop _{\Omega\cap Q} \eta \Bigl(\frac{y-w}{\epsilon} + \alpha e_n \Bigr) \wrt w \right) \\
 & \leq & C^2.
\end{eqnarray*}
At last, since $ \mathrm{spt}(g) \subset Q'  $ it follows that $ \mathrm{spt}(g_\epsilon) \subset Q'  $  for sufficiently small $ \epsilon  $ and thus we can extend it by $ 0 $ to the rest of $ \overline{\Omega} $. Note that the convergence of $ g_\epsilon  $ to $ g $ -- in $ W^q (\Omega,\dive) $, in every Lebesgue point $ y\in \Omega \setminus Q' $  and uniformly on $ \overline{\Omega} $ in the case that additionally $ g\in C(\overline{\Omega},\mathbb{R}^d) $ -- and also the uniform boundedness on all of $ \overline{\Omega} $ are trivially satisfied.

In the second step we make use of the previous calculations to get an approximation to $ g $ without additional assumptions: Since $ \partial \Omega $ is compact, there exist finitely many cubes $ Q'_i= Q_{\frac{r_i}{2}} (x_i),\,1\leq i \leq M $ as above, which cover $ \partial \Omega  $. Let $ (\zeta _i)_{0\leq i \leq M} $ be $ C^\infty $-functions, such that
\begin{equation*}
\left\{ 
\begin{gathered}
\shoveleft{0\leq \zeta _i \leq 1 \quad \mathrm{spt}(\zeta _i )\subset Q_i ' \quad \mathrm{for } \, 1\leq i\leq M,} \\
\shoveleft{ 0\leq \zeta _0 \leq 1 \quad \mathrm{spt}(\zeta _0 )\subset \Omega ,} \\
\shoveleft{ \sum _{i=0} ^M \zeta _i \equiv 1 \quad \mbox{on }  \Omega .}
\end{gathered}
\right.
\end{equation*}
As shown above, for $ g\zeta _i $, $ 1\leq i \leq M $ we can construct vector fields $ g_{\epsilon,i} \in C^\infty (\overline{\Omega},\mathbb{R}^d)$ converging to $ g\zeta _i $ in the desired sense. By a standard mollifier approximation we can also construct $ g_{\epsilon,0} $ converging to $ g\zeta _0 $ in the desired sense. Setting 
\begin{equation*}
g_\epsilon = \sum _{i=0}^M g_{\epsilon,i}
\end{equation*}
we finally obtain vector fields in $ C^\infty (\overline{\Omega},\mathbb{R}^d) $ converging to $ g $ in $ W^q (\dive ;\Omega) $ as $ \epsilon \rightarrow 0 $ and, as one can check easily, satisfying also the additional boundedness and  convergence properties 2), 3), 4).
\end{proof}
\end{appendix}

\bibliography{lit_dat}
\bibliographystyle{plain}

\end{document}